\documentclass[12pt,a4paper]{amsart} 

\usepackage{graphicx}
\usepackage{amsmath, amssymb,amsfonts,amsthm}
\usepackage{tikz,tabularx}
\usepackage{geometry}
\usepackage{natbib}
\usepackage[normalem]{ulem}
\usepackage[foot]{amsaddr}
\usetikzlibrary{arrows,positioning,decorations,matrix}

\newcommand{\C}{\mathcal C}

\newcommand{\newl}{\ell}
\newcommand{\al}{\alpha}
\newcommand{\pionesub}[1]{\pi_1^{(#1)}}
\newcommand{\pitwosub}[1]{\pi_2^{(#1)}}
\newcommand{\Dbg}{\mathcal{D}}
\newcommand{\eg}{{\it Example }}
\newcommand{\op}[1]{\mathrel{\mathop{\rightarrow}^{\mathrm{#1}}}}

\newtheorem{definition}{Definition}[section]
\newtheorem{lemma}{Lemma}[section]
\newtheorem{corollary}{Corollary}[section]
\newtheorem{theorem}{Theorem}[section]
\title[Algebraic double cut and join]{Algebraic double cut and join -- A group-theoretic approach to the operator on multichromosomal genomes}

\author{Sangeeta Bhatia \and Attila Egri-Nagy \and Andrew R. Francis}
\address{Centre for Research in Mathematics\\
University of Western Sydney\\
Australia}
\email{s.bhatia@uws.edu.au,a.egri-nagy@uws.edu.au,a.francis@uws.edu.au}


\begin{document}

\maketitle

\begin{abstract}
Establishing a distance between genomes is a significant problem in computational genomics, because its solution can be used to establish evolutionary relationships including phylogeny. 

The ``double cut and join'' (DCJ) model of chromosomal rearrangement proposed by \citet{Yancopoulos2005} has received attention as it can model inversions, translocations, fusion and fission on a multichromosomal genome that may contain both linear and circular chromosomes.
In this paper, we realize the DCJ operator as a group action on the space of multichromosomal genomes.  We study this group action, deriving some properties of the group and finding group-theoretic analogues for the key results in the DCJ theory.

\end{abstract}

\section{Introduction}
\label{sec:intro}

The use of genome rearrangements to estimate evolutionary distance dates back as far as \citet{Watterson-chrom-reversal-1982}.  The novelty of this approach is that it ignores single nucleotide polymorphisms (SNPs) and takes genes and their relative positions (and orientation) as the fundamental unit of DNA for the purposes of distance calculations.  This is particularly valuable in the case of bacterial DNA, which undergoes relatively frequent rearrangement (relative to eukaryotes, that is), but which also experiences significant horizontal, or lateral, gene transfer (HGT).  The use of large-scale rearrangements to establish distance is thought to be less vulnerable to the effects of HGT than the use of SNPs~\citep{Darling2008}.

The \emph{double cut and join} (DCJ) operator, introduced by \citet{Yancopoulos2005} (see also \citet{Bergeron2006b}), provided a significant breakthrough by treating a much larger family of operations acting on a more general, multichromosomal genome, and showing how distance can be expressed in a very simple formula based on features of a graph derived directly from the genome arrangements.  While the DCJ operator treats all operations in its remit as equally likely, it is possible that this operator may provide a valuable base for operators that account for differences in frequency of different operations, and may potentially be specialised to genomes with specific chromosomal structure (such as a single circular chromosome).  
 
With this potential in mind, in this paper we translate the DCJ operator into a group-theoretic setting.    
We show that by expressing the multi-chromosomal genome with $n$ oriented regions as a permutation of $\{1,\dots,2n\}$, a DCJ operator can be defined as an action on the genome.  Hence, the set of double cut and join operators generates a group acting on the entire genome space.  The DCJ distance is then a path distance on the Cayley graph of this group (as described in~\citet{egrinagy2013group}).  We show how the DCJ distance between two genomes can be obtained in a very simple way from the permutation encoding of the genomes.  We obtain a formulation for the DCJ distance that is analogous to the distance formula found by~\citet{Yancopoulos2005}, but is expressed in terms of features of a permutation.  This is derived independently of the established DCJ theory.

Over the last decade, there have been several examples of  algebraic approaches to modeling biological phenomena, particularly in the genomic distance literature.  While the traditional approach to the problem of finding distance between genomes is to cast them as permutations, limited use has been made of the powerful machinery that algebra provides to deal with permutations. Adopting an algebraic viewpoint might in fact reveal deep insights and lead to simplification.

A recent example of this is the work of \citet{lu2006analysis} who used the theory of symmetric groups to give an algorithm that gives a sorting sequence between circular genomes using fission, fusion and block interchanges.  More recently, group theory has been used by the authors' group to calculate the inversion distance between circular genomes under one model~\citep{egrinagy2013group}, and a wider algebraic framework has been proposed that includes DNA knotting~\citep{francis2013algebraic}.
 
A circular genome is modeled as a cyclic permutation by \citet{meidanis2000alternative}, treating a genome as a permutation of the genes $a_1,a_2, \hdots, a_n$. Writing the circular genome in cycle notation as $(a_1,a_2,\hdots,a_n)$ denotes that gene $a_{i}$ is adjacent on the genome to gene $a_{i+1}$, with gene $a_n$ being adjacent to $a_1$. This approach allowed them to derive many important properties of the breakpoint graph in terms of permutation products, and to give a lower bound on the transposition distance. This work was later extended in \citet{Feijao2013} to include linear chromosomes. \citet{Feijao2013} model a genome as a product of disjoint 2-cycles, and present a formulation of a $k$-break operation as a permutation. The double cut and join model then becomes a special case of the $k$-break operation with $k=2$.

The model of the genome as a product of disjoint 2-cycles and the double cut and join operation as conjugation used by \citet{Feijao2013}, is also employed by us in this paper. The novelty of our work lies in development of this model in a completely algebraic framework, without making use of the existing theory. This allows us to present new proofs for existing results, sometimes leading to a considerable simplification of arguments such as in the result related to counting sorting scenarios (Theorem~\ref{thm:sortingScen}).

This paper is organized as follows. In Section~\ref{sec:dcj} we introduce the double cut and join model, following the exposition of~\citet{Bergeron2006b}.  Section~\ref{sec:genPerms} explains how we can encode a genome as a product of 2-cycles, essentially extending the concept of ``adjacencies" and ``telomeres" from the established theory.  In Section~\ref{sec:algebraicDCJ} we give the major construction of this paper, namely the definition of the DCJ operator as a group action on the genome space.   While the standard definition has several cases depending on the arrangement of the genome, this action requires just two cases.
In order to find a distance formula in this model (Section~\ref{sec:dcjDisNew}), we first need to establish some results about products of involutions, covered in Section~\ref{sec:invlProd}.  The main result, establishing distance in this model, is given in the following theorem:

\medskip
\noindent
\textbf{Main Theorem} (Theorem~\ref{thm:dcjDisTotal})
\emph{Let $G_1$ and $G_2$ be genomes on $n$ regions with corresponding  genomic permutations  $\pi_1$ and $\pi_2$. The DCJ distance between $G_1$ and $G_2$ is given by
\[d_{DCJ}\left(\pi_1,\pi_2 \right)=\frac{1}{2}\left(\newl_t(\pi_2\pi_1)+ n_c\right)\]
where $n_c$ is the number of cycles in the product $\pi_2\pi_1$ which contain two fixed points of ${\pi_1}$ or ${\pi_2}$, and $\newl_t$ is the transposition length.
}

Finally, in Section~\ref{sec:numb.scenarios}, we derive a formula for the number of optimal sorting scenarios between two genomes.  That is, the number of minimal length paths in the Cayley graph of the group generated by the DCJ operators. As our work utilizes many well-known results about permutations, we have collected them in Appendix~\ref{sec:gt_primer} for ease of reference. These results are stated without proofs. A complete treatment may be found in an abstract algebra text such as \citep{herstein2006topics,fraleigh2003first}.

\section{The double cut and join model}
\label{sec:dcj}

In this section we follow the notation of \citet{Bergeron2006b}.  For a more complete introduction to the model and results, see that paper as well as \citet{Yancopoulos2005}.

\subsection{The genome graph}
\label{subsec:genGraph}
Before presenting the double cut and join operator, we first explain how multichromosomal genomes are modeled. In this model, a gene is essentially an oriented section of the DNA and its two ends are called its \emph{extremities}. This of course differs from the biological meaning of the word gene and is closer to what are referred to as ``conserved blocks'' in the rearrangement literature (for example see \citet{hannenhalli1995transforming}, \citet{lin2006exposing}). However, for convenience we will use the words gene and region interchangeably. The extremities of the gene $a$ are denoted $a_t$ and $a_h$ where the subscripts stand for tail and head respectively. 

To represent a genome,  considered as an arrangement of oriented genes, it is sufficient to note which extremities are adjacent on the genome. An extremity that is not adjacent to any other is the end point of a linear section of the genome and is called a \emph{telomere}.
An (unordered) pair of extremities that are adjacent on the genome is referred to as an \emph{adjacency}. For instance, the adjacency $\{a_t,b_h\}$ indicates that the tail of gene $a$ is adjacent to the head of gene $b$ on the genome. Note that an extremity can be adjacent to at most one other extremity.  

Thus, in this model a genome is represented by a partition of the set of extremeties of the genes into subsets of cardinality 1 (telomeres) or 2 (adjacencies). Equivalently, the genome can be viewed as a graph whose vertex set is the set of all adjacencies and telomeres and whose edges are drawn between the extremities of the same gene. Thus every vertex of a genome graph has degree one or two. Figure~\ref{fig:genome} illustrates a genome graph.

\begin{figure*}[ht]

\begin{center}
\begin{tikzpicture}[label distance=1pt]
  \draw [-, thick] (0,0) -- (8,0) node (gene1) {};

  \draw [fill] (0,0) circle [radius=.05] node [label=above:$1_t$,label=below:$(1)$] {};
  \draw [fill] (2,0) circle [radius=.05] node [label=above:${1_h,3_t}$,label=below:${(2,5)}$] {};
  \draw [fill] (4,0) circle [radius=.05] node [label=above:${3_h,2_t}$,label=below:${(6,3)}$] {};
  \draw [fill] (6,0) circle [radius=.05] node [label=above:${2_h,4_t}$,label=below:${(4,7)}$] {};
  \draw [fill] (8,0) circle [radius=.05] node [label=above:$4_h$,label=below:$(8)$] {};
  
  \draw [-, thick] (10,1) arc [radius=1, start angle=90, end angle=-90] ;
  \draw [fill] (10,1) circle [radius=.05] node [label=above:{$5_h,6_t$},label=below:${(10,11)}$] {} ;
  \draw [-, thick] (10,-1) arc [radius=1, start angle=-90, end angle=-270] ;
  \draw [fill] (10,-1) circle [radius=.05] node [label=below:{$5_t,6_h$},label=above:${(9,12)}$] {};	 
  
\end{tikzpicture}
\end{center}
\caption{The genome graph of a genome with one linear chromosome containing genes numbered 1, 2, 3 and 4, and one circular chromosome containing genes numbered 5 and 6. The vertex set of the graph is $\{\{1_t\},\{1_h,3_t\},\{3_h,2_t\},\{2_h,4_t\},\{4_h\},\{5_h,6_t\},\{6_h,5_t\}\}$. Edges are drawn between extremities of the same gene. \newline 
The map $\phi$ (Definition~\ref{def:phi}) maps the set of extremities $\{1_t,1_h,2_t,2_h, \hdots, 6_h\}$ into $ \{1,2,3,4, \hdots, 12 \}$ , with $\phi(1_t)= 1$, $\phi(1_h)=2$ and so on. $1_t$ and $4_h$ are telomeres, hence $1$ and $8$ are fixed points of the permutation $\pi$ (Definition~\ref{def:pi}). $1_h$ is connected to $3_t$ which is captured by the 2-cycle $(2,\ 5)$ in the permutation encoding. The other 2-cycles can be similarly interpreted. The above genome is thus encoded as the permutation $(2,5)(3,6)(4,7)(10,11)(9,12)$.
}
\label{fig:genome}
\end{figure*}
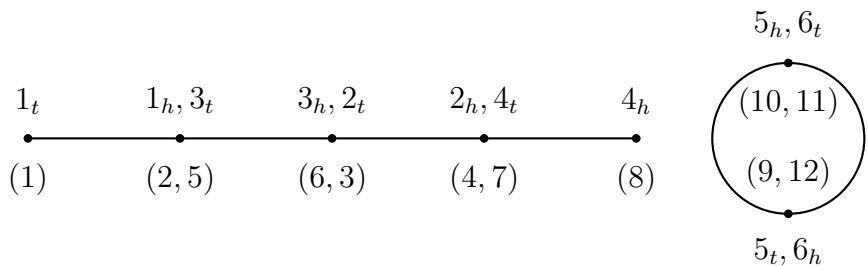

\subsection{The double cut and join operator}
\label{subsec:dcjOp}

The \emph{double cut and join} (DCJ) operator acts on a pair of vertices of a genome graph in one of the following ways:

\begin{enumerate}
\item $\{p,q\},\{r,s\}$ may be changed to $\{p,r\},\{q,s\}$ or $\{p,s\},\{q,r\}$,

\item $\{p,q\},\{r\}$ may be changed to $\{p,r\},\{q\}$ or $\{q,r\},\{p\}$,

\item $\{p,q\}$ may be changed to $\{p\},\{q\}$ or $\{p\},\{q\}$ changed to $\{p,q\}$.
\end{enumerate}

Depending on the vertices that it acts on, the double cut and join operator can simulate the inversion, excision and translocation of a section of the genome as well as fusion and fission of chromosomes. Figure~\ref{fig:dcjExample} presents some examples.
\begin{figure*}[ht]
\centering
 \begin{tikzpicture}[label distance=1pt]
   \matrix[column sep=3em,row sep=3em]{
   \node {(a)};
   &
   \draw [-,thick] (0,0) -- (4,0); 
   \draw [fill] (0,0) circle [radius=.05] node [label=above:{$1_t$}] {};
   \draw [fill] (1.5,0) circle [radius=.05] node [label=above:{$1_h,3_t$}] {};
   \draw [fill] (3,0) circle [radius=.05] node [label=above:{$3_h,2_t$}] {};
   \draw [fill] (4,0) circle [radius=.05] node [label=above:{$2_h$}] {};
   &
   \node {$\longrightarrow$};
   &
   \draw [-,thick] (0,0) -- (4,0); 
   \draw [fill] (0,0) circle [radius=.05] node [label=above:{$1_t$}] {};
   \draw [fill] (1.5,0) circle [radius=.05] node [label=above:{$1_h,3_h$}] {};
   \draw [fill] (3,0) circle [radius=.05] node [label=above:{$3_t,2_t$}] {};
   \draw [fill] (4,0) circle [radius=.05] node [label=above:{$2_h$}] {};\\

   \node {(b)};
   &
   \draw [-,thick] (0,0) -- (4,0); 
   \draw [fill] (0,0) circle [radius=.05] node [label=above:{$1_t$}] {};
   \draw [fill] (1.5,0) circle [radius=.05] node [label=above:{$1_h,3_t$}] {};
   \draw [fill] (3,0) circle [radius=.05] node [label=above:{$3_h,2_t$}] {};
   \draw [fill] (4,0) circle [radius=.05] node [label=above:{$2_h$}] {};
   &
   \node {$\longrightarrow$};
   &
   \draw [-,thick] (0,0) -- (4,0); 
   \draw [fill] (0,0) circle [radius=.05] node [label=above:{$1_h$}] {};
   \draw [fill] (1.5,0) circle [radius=.05] node [label=above:{$1_t,3_t$}] {};
   \draw [fill] (3,0) circle [radius=.05] node [label=above:{$3_h,2_t$}] {};
   \draw [fill] (4,0) circle [radius=.05] node [label=above:{$2_h$}] {};\\

  \node {(c)};
  &
   \draw [-,thick] (0,0) -- (3,0); 
   \draw [fill] (0,0) circle [radius=.05] node [label=above:{$1_t$}] {};
   \draw [fill] (1,0) circle [radius=.05] node [label=above:{$1_h,3_t$}] {};
   \draw [fill] (2,0) circle [radius=.05] node [label=above:{$3_h,2_t$}] {};
   \draw [fill] (3,0) circle [radius=.05] node [label=above:{$2_h$}] {};

   \draw [-,thick] (4,0) -- (5,0);
   \draw [fill] (4,0) circle [radius=.05] node [label=above:{$4_h$}] {};
   \draw [fill] (5,0) circle [radius=.05] node [label=above:{$4_t$}] {};
      
  &
  \node {$\longrightarrow$};
  &
   \draw [-,thick] (0,0) -- (4,0);  
   \draw [fill] (0,0) circle [radius=.05] node [label=above:{$1_t$}] {};
   \draw [fill] (1,0) circle [radius=.05] node [label=above:{$1_h,3_t$}] {};
   \draw [fill] (2,0) circle [radius=.05] node [label=above:{$3_h,2_t$}] {};
   \draw [fill] (3,0) circle [radius=.05] node [label=above:{$2_h,4_h$}] {};
   \draw [fill] (4,0) circle [radius=.05] node [label=above:{$4_t$}] {};\\
 };
 \end{tikzpicture} 

\caption{(a) $\{1_h,3_t\},\{3_h,2_t\}$ is changed to $\{1_h,3_h\} ,\{3_t,2_t\} $ leading to an inversion. (b) $\{1_t\} ,\{1_h,3_t\}$ is changed to $\{1_h\} ,\{1_t,3_t \} $, another inversion. (c) $\{2_h\} ,\{4_h\} $ is changed to $\{2_h,4_h\} $, a fusion. }
\label{fig:dcjExample}
 \end{figure*}
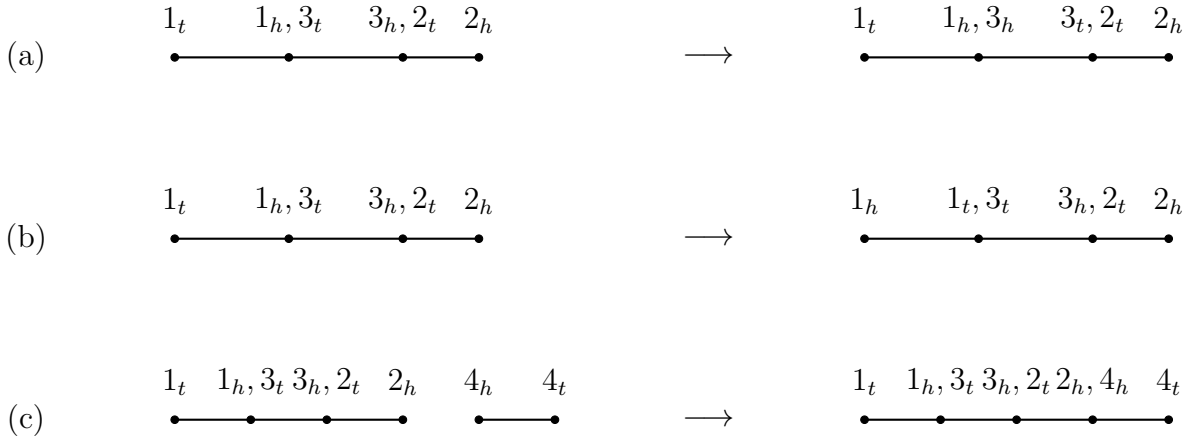	
%---------------------------------------------%
\subsection{The double cut and join distance}
\label{subsec:dcjDis}

The DCJ distance between genomes $G_1$ and $G_2$ is the minimal number of DCJ operations required to change one genome into the other. 

\citet{Bergeron2006b} make use of a graph construct called the ``adjacency graph'' to determine the DCJ distance between two genomes. An \emph{adjacency graph} $AG(G_1,G_2)$ can be drawn for any pair of genomes $G_1$ and $G_2$ defined on the same set of $n$ genes. The vertex set of the graph is the set of all adjacencies and telomeres in $G_1$ and $G_2$. For an adjacency (or telomere) $u \in G_1$ and adjacency (or telomere)  $v \in G_2$, there is an edge in $AG(G_1,G_2)$ between $u$ and $v$ for each gene extremity they have in common.

The vertices of the adjacency graph then have degree either one (at a telomere) or two (at an adjacency), so the graph consists of a set of cycles and a set of paths.  Let $c$ be the number of cycles and $p$ be the number of paths of odd length in $AG(G_1,G_2)$.
\citet{Bergeron2006b} established that the DCJ distance between two genomes can be given in terms of these adjacency graph {statistics} as follows:

\[d_{DCJ}\left(G_1,G_2\right)=n-(c+p/2).\]

\section{Genomes as permutations}
\label{sec:genPerms}

We now present our reformulation of the double cut and join model. We first formalize the notion of a genome on $n$ regions. Let $\{h,t\}$ be the extremities of a gene where $h$ and $t$ denote the head and tail respectively. Let $\mathbf{n}$ be the set $\{1,2,\hdots,n\}$ enumerating the $n$ regions.

\begin{definition}[Extremities]
The Cartesian product $E=\mathbf{n} \times \{h,t\}$ is the set of all extremities of $n$ regions.
\end{definition}

To conform to the notation used earlier in this paper and in previous literature, we will use $i_h$ and $i_t$ to denote the extremities $(i,h)$ and $(i,t)$ giving the head and tail of gene $i$ respectively.

We define a map that assigns numeric labels to the elements of $E$.

\begin{definition}[Assignment map] \label{def:phi} Let $\phi:E \rightarrow \mathbf{2n}$ be defined as follows:
\[\phi(i_t)=2i-1,\]
\[\phi(i_h)=2i.\]
\end{definition}
%use i_e to inidcate (i,e).
\begin{definition}[Genome] \label{def:pi} A \emph{genome} on $n$ regions is a permutation $\pi$ on the set $E$ such that
\[\pi(i)=j \iff \pi(j)=i.\]
\end{definition}

The above definition implies that a genomic permutation is a product of disjoint $2$-cycles. The restriction in the definition of a genome captures the notion of pairing of gene extremities on a genomic strand. Therefore a $2$-cycle in this formulation is an adjacency, and similarly, fixed points of a permutation are telomeres. It is important to note that at this point we use the permutation $\pi$ as a static description of the genome, not as an operation, so that the $2$-cycles can be considered  as synonyms for unordered pairs. Furthermore, this construction with 2-cycles representing adjacencies means that the identity permutation will only arise in the trivial case in which all chromosomes in the genome contain just a single region.

As mentioned in section~\ref{sec:dcj}, the vertex set of the genome graph consists of the adjacencies and telomeres. For every gene $i$, an edge is drawn between the adjacency containing $i_h$ and $i_t$. In writing the genome as a permutation, the $2$-cycles and the fixed points are the adjacencies and the telomeres. The assignment map $\phi$ tells us the correspondence between the gene extremities and the set  $\mathbf{2n}$. Hence the assignment map $\phi$ and the genomic permutation $\pi$ contain all the information that is needed to construct the genome.

\citet{bafna1993genome} introduced the notion of a breakpoint graph for an unsigned permutation. To extend this concept to signed permutations, they transform a signed permutation $\pi$ on $n$ elements to an unsigned permutation $\pi^{\prime}$ on $2n$ elements. This is done by replacing a positive integer $i$ in $\pi$ by $2i-1$ followed by $2i$ in $\pi^{\prime}$ and by replacing a negative integer $-i$ in $\pi$ by $2i$ followed by $2i-1$.  This transformation is precisely the labeling map $\phi$. Reverse orientation of a gene on a chromosome means that the tail of the gene is present after the head of the gene. Thus in the permutation representation, a negative integer $-i$ is replaced by $2i$ (label of $i_h$) followed by $2i-1$ (label of $i_t$).

We use cycle notation to write permutations.  Thus the cycle $(i_1, i_2, i_3,\hdots , i_n)$ in a permutation $\alpha$ means that $\alpha(i_1)=i_2,\alpha(i_2)=i_3$ etc. and $\alpha(i_n)=i_1$ (see Appendix~\ref{sec:gt_primer}).  Figure~\ref{fig:genome} illustrates an example of permutation encoding of a genome on $6$ regions. 

A genome on $n$ regions is a permutation of the set $\mathbf{2n}$ satisfying the constraints in definition~\ref{def:pi}. Lemma~\ref{lemma:genCount} gives an expression for the number of permutations in $S_{2n}$ satisfying this definition.

\begin{lemma}\label{lemma:genCount}
The number of genomes on $n$ regions is given by 
\[\sum_{t=0}^n{\binom{2n}{2t}(2n-2t-1)!!}=\sum_{t=0}^n{\binom{2n}{2t}(2t-1)!!}.\]
\end{lemma}
\begin{proof}
Let the set of all genomes on $n$ regions be $\Gamma_n$. Each genome is a permutation of the $2n$extremities $E$, and hence $\Gamma_n$ is a subset of the symmetric group $S_{2n}$. The cardinality of $\Gamma_n$ can be determined as follows. 

Each genomic permutation can have an even number of fixed points, since it is a product of disjoint 2-cycles and 1-cycles acting on a set of even cardinality ($2n$). This also follows from the fact that fixed points are telomeres, and a genome must have an even number of telomeres.

Let the number of fixed points be $2t$. The remaining $2n-2t$ elements must be paired off with each other. Each such pairing of the $2n-2t$ elements defines an involution  in the symmetric group $S_{2n-2t}$ that does not have any fixed points.  An involution is an element of order $2$ i.e., $\pi$ is an involution if $\pi^2$ is the identity permutation.

The number of such involutions is $(2n-2t-1)!!$ \citep[pp. 15-16]{Stanley1999} where the double factorial function is the product of odd numbers i.e. $(2k-1)!!=\prod_{i=1}^k(2i-1)$. Therefore the cardinality of $\Gamma_n$ is given by
\[\left\vert \Gamma_n \right\vert = \sum_{t=0}^n{\binom{2n}{2t}(2n-2t-1)!!}=\sum_{t=0}^n{\binom{2n}{2t}(2t-1)!!}.\]
\qed

\end{proof}
The number of genomes is already almost a billion for 9 regions.  The astute observer will note that this number is also the number of tableaux on $2n$ elements, with a correspondence given by the Robinson-Schensted algorithm (see for instance \cite{fulton1997young}).  The first nine numbers in the sequence are shown in Table~\ref{tab:tableux}. 
\begin{table}
\begin{center}
\begin{tabular}{|r|r|}
\hline
\#regions & \#genomes\\
\hline
1& 2\\
\hline
2& 10\\
\hline
3& 76\\
\hline
4& 764\\
\hline
5& 9496\\
\hline
6& 140152\\
\hline
7& 2390480\\
\hline
8& 46206736\\
\hline
9& 997313824\\
\hline
\end{tabular}
\end{center}\caption{The number of genomes on $n$ regions (also the number of tableaux on $2n$ elements).}\label{tab:tableux}
\end{table}
%=============================================%
\section{The DCJ operator as an action on a permutation}
\label{sec:algebraicDCJ}
%=============================================%

In this section we define an algebraic version of the DCJ operator acting on the set $\Gamma_n$ of genomes on $n$ regions, and show that it is an involution.  Appendix~\ref{sec:gt_primer} contains a summary of some results on symmetric groups that may be useful for reference in this section and the next.

As explained in the previous section, the genome is modeled as a set of unordered pairs of gene extremities (adjacencies) and single gene extremities (telomeres). A DCJ operation as defined in \cite{Bergeron2006b} swaps gene extremities between two pairs (i.e. adjacencies) or a pair and a singleton, as described in Section~\ref{subsec:dcjOp}.

Hence, the possible scenarios are that the two gene extremities being swapped are: adjacent to each other on the genome; both  involved in different adjacencies; one of them is a telomere and the other in an adjacency; or both of them are telomeres.  When a DCJ operation acts on a pair of gene extremities that form an adjacency, it splits them, producing two telomeres, and conversely when it acts on two telomeres, it combines them into an adjacency. In the two other cases, an extremity is swapped.

Thus, in the permutation representation, the DCJ operation swapping $i$ and $j$ changes:

\begin{align*}
(i,k)(j,l)\quad&\longrightarrow\quad(j,k)(i,l),\\
(i,k)(j)\quad&\longrightarrow\quad(j,k)(i),\\
(i,j)\quad&\longrightarrow\quad(i)(j),\quad\text{ and}\\
(i)(j)\quad&\longrightarrow\quad(i,j). 
\end{align*}

With this in mind, for $i,j \in \mathbf{2n}$, and $i \neq j$ we define the double cut and join operator $D_{ij}$ acting on the set of genomes $\Gamma_n$, $D_{ij}: \Gamma_n \rightarrow \Gamma_n $ as follows:
\begin{definition}[Set of fixed points] Let $\pi$ be a permutation of $\mathbf{2n}$, then the set of \emph{fixed points} of $\pi$ is defined by
\[ F_{\pi}:=\{i \mid i \in \mathbf{2n}, \pi(i)=i\}. \]
\end{definition}

\begin{definition}[Algebraic double cut and join operator]\label{d:Dij}
For a permutation $\pi$ representing a genome, set % as described above
\[
D_{ij}(\pi): = 
	\begin{cases}
 	  (i,j)\pi & \text{if }i,j \in F_{\pi}\text{ or }\pi(i)=j,\text{ and}\\	
	  (i,j)\pi(i,j) & \text{otherwise. }
	\end{cases}
\]

\end{definition}

Therefore, in algebraic terms, the double cut and join operators as defined above are conjugations or left actions by $2$-cycle involutions.
To distinguish this formulation from the standard, we will call $D_{ij}$ the \emph{algebraic} double cut and join operator.

The vertex set of the genome graph consists of unordered 2-tuples and singletons from the set of gene extremities $E$.  The map $\phi$ simply  relabels  elements of the set $E$ with the labels from the set $\mathbf{2n}$. Therefore we can consider the vertex set of the genome graph to consist of unordered 2-tuples and singletons from $\mathbf{2n}$. A genomic permutation $\pi$ is a permutation on the set $\mathbf{2n}$ satisfying the constraint $\pi(i)=j \iff \pi(j)=i$.
 
 Let $\rho$ be the map that writes the vertex set of the genome graph $G$ as permutation $\pi$. % 
 \[\rho(\{i,j\})=(i,j) \textit{ and } \rho(\{i\})=(i).\]
 
 Let $\Dbg_{\{i,j\}}(G)$ be the DCJ operator acting on the extremities $i$ and $j$ of genome $G$ and let $D_{ij}(\pi)$ be the operator acting on the permutation $\pi$. The remarks motivating the definition of  algebraic DCJ operator informally explain why we can expect the graph-theoretic and the algebraic operators to be equivalent. That is, the diagram in Figure~\ref{fig:commutative} commutes.
 We now prove this statement formally.
 
 \begin{figure}[h]
 \begin{center}
 \begin{tikzpicture}
 \matrix (m) [matrix of math nodes,row sep=3em,column sep=4em,minimum width=2em]
 {
 G & \pi \\
 G^\prime & \pi^\prime \\ };
 \path[-stealth] (m-1-1) edge node [left] {$\Dbg_{\{i,j\}}$} (m-2-1);
 \path[-stealth] (m-2-1) edge[<->] node [below] {$\rho$} (m-2-2);
 \path[-stealth] (m-1-2) edge node [right] {$D_{ij}$} (m-2-2);
 \path[-stealth] (m-1-1) edge[<->] node [above] {$\rho$} (m-1-2);
 \end{tikzpicture}
  \end{center}
 \caption{Rewriting genome $G$ as the permutation $\pi$, and acting on $\pi$ by $D_{ij}$ gives the same result as acting on $G$ by the DCJ operator and then rewriting the result as permutation $\pi^{\prime}$. }
 \label{fig:commutative}
 \end{figure}

\begin{lemma} For all genomes $G$, 
	 \[ \rho  \left(\Dbg_{\{i,j\}}(G)\right) = D_{ij} \left(\rho(G) \right).\]
\end{lemma}
\begin{proof}	
We prove this by considering all the four cases in the definition of the operator $\Dbg$ (Section~\ref{subsec:dcjOp}).

Case 1. $i$ and $j$ are in separate 2-tuples $\{i,k\},\{j,l\} \in G$.

\[\rho  \left(\Dbg_{\{i,j\}}(G)\right)=\rho\left(\Dbg(\{i,k\},\{j,l\})\right)=\rho \left(\{j,k\},\{i,l\}\right)=(j,k)(i,l).\]

\[D_{ij} \left( \rho(\{i,k\},\{j,l\} )  \right)= D_{ij} \left((i,k)(j,l)\right)=(j,k)(i,l).\]

Case 2. Exactly one of the $i,j$ is in a 2-tuple $\{i,k\},\{j\} \in G$.

\[\rho  \left(\Dbg_{\{i,j\}}(G)\right)=\rho\left(\Dbg(\{i,k\},\{j\})\right)=\rho(\{(j,k)\},\{i\})=(j,k)(i).\]

\[D_{ij} \left( \rho(\{i,k\},\{j\} )  \right)= D_{ij} \left((i,k)(j)\right)=(j,k)(i).\]

Case 3. None of the $i,j$ is in a 2-tuple. $\{i\},\{j\} \in G$.
\[\rho  \left(\Dbg_{\{i,j\}}(G)\right)=\rho  \left(\Dbg_{\{i,j\}}(\{i\},\{j\})\right)=\rho(\{i,j\})=(i,j).\]

\[D_{ij} \left( \rho(\{i\},\{j\} )  \right)= D_{ij} \left((i)(j)\right)=(i,j).\]

Case 4. $i,j$ are in the same 2-tuple in $G$ $\{i,j\} \in G$
\[\rho  \left(\Dbg_{\{i,j\}}(G)\right)=\rho  \left(\Dbg_{\{i,j\}}(\{i,j\})\right)=\rho(\{i\},\{j\})=(i)(j).\]

\[D_{ij} \left( \rho(\{i,j\} )  \right)= D_{ij} \left((i,j)\right)=(i)(j).\]

Thus in all cases, $\rho  \left(\Dbg_{\{i,j\}}(G)\right) = D_{ij} \left( \rho(G)  \right)$.

\qed
\end{proof}

The following lemma shows that the algebraic DCJ operator is an involution.

\begin{lemma}
$D_{ij}^2(\pi)=\pi$ {for all } $\pi \in \Gamma_n$ and $i,j \in \mathbf{2n}$.
\end{lemma}

\begin{proof}

For any permutation $\pi \in \Gamma_n$, if $i,j$ are not both telomeres and do not form an adjacency of $\pi$, then the same holds in $D_{ij}(\pi)=(i,j)\pi(i,j)$. Similarly if $i$ and $j$ are both telomeres in $\pi$ then they form an adjacency in $D_{ij}(\pi)$, and if they are in an adjacency in $\pi$, they will both be telomeres in $D_{ij}(\pi)$. Thus acting by $D_{ij}$ on $D_{ij}(\pi)$ will cause the same condition in the definition of $D_{ij}$ to be invoked which was invoked when $D_{ij}$ acted on $\pi$.

The operation in each case is an involution, hence $D_{ij}^2(\pi)=\pi$ {for all } $\pi \in \Gamma_n$.
\qed
\end{proof}

At this point, we make the following note about the notation employed in the remainder of this paper. Permutations are functions where the operand is written on the right. So for example, $\pi(i)$ is the permutation $\pi$ acting on $i$. In line with this, permutation multiplication is done from right to left. 

A $2$-cycle is written as $(i,j)$. We will not in general write cycles of length $1$, except for emphasis.
%=============================================%
\section{Products of involutions}
\label{sec:invlProd}
%=============================================%
In this section we prove some results about products of involutions. We make use of these results in Section~\ref{sec:dcjDisNew} to determine the DCJ distance between genomic permutations $\pi_1$ and $\pi_2$.

\begin{lemma}
\label{lemma:fpfree}
Let $\al$ and $\beta$ be involutions acting on the set $\mathbf{2n}=\{1,2,\hdots,2n\}$.  If $F_\al=F_\beta=\varnothing$, then
\begin{enumerate}
\item For any $i \in \mathbf{2n}$, $i$ and $\al(i)$ are in different cycles of $\beta \al$. Similarly $i$ and $\beta(i)$ are in different cycles of $\beta \al$.
\item $\beta \al$ has an even number of cycles of length $k$ for any $k \in \mathbb{N}$.
\end{enumerate}
\end{lemma}

\begin{proof} 
(1) The cycle in $\beta \al$ containing $1$ is of the form

\[\left( 1,\beta \al(1),\beta \al\beta\al(1),\hdots ,(\beta \al)^k(1)\right),\]

where $k$ is the smallest positive integer such that $(\beta \al)^{k+1}(1)=1$, therefore the length of this cycle is $k+1$.
We claim that $\al(1) \notin \{1,\beta \al(1),\hdots, (\beta \al)^k(1)\}$.

Suppose that $\al(1)=(\beta\al)^r(1)$ for some $r$. If $r$ is even then
\begin{align*}
\al(1)&=(\beta\al)^r(1) \\
&= (\beta\al)^{(r/2)-1}\beta\al(\beta\al)^{r/2}(1).
\end{align*}
By multiplying on the left both sides by $(\al\beta)^{(r/2)-1}$, (the inverse of $(\beta\al)^{(r/2)-1}$), we get 
\[(\al\beta)^{(r/2)-1}\al(1)=\beta\al(\beta\al)^{r/2}(1),\]
and multiplying by $\beta$ yields 
\begin{align*}
\beta(\al\beta)^{(r/2)-1}\al(1)&=\al(\beta\al)^{r/2}(1)\\
(\beta\al)^{r/2}(1)&=\al\left((\beta\al)^{r/2}(1)\right).
\end{align*}

In other words, $(\beta\al)^{r/2}(1)\in F_\al$, contradicting the assumption that $F_\al=\varnothing$. 
Similarly, if $r$ is odd, we find that $\al(1)=(\beta\al)^{r}(1)$ implies that $(\al\beta)^{(r+1)/2}\al(1)$ is a fixed point of $\beta$, another contradiction.

Thus $\al(1) \notin \{1,\beta \al(1),\hdots, (\beta \al)^k(1)\}$.

(2) Write the cycle in $\beta \al $ containing $\al(1)$ as

$\left( \al(1),\beta(1),\beta\al\beta(1),\hdots ,(\beta\al)^s\beta(1) \right)$,

where $s$ is the smallest positive integer such that $(\beta\al)^{s+1}\beta(1)=\al(1)$. Then multiplying both sides of the equation by $\beta (\al\beta)^{s+1}$ (the inverse of $(\beta\al)^{s+1}\beta$) we obtain
\begin{align*}
1&= \beta (\al\beta)^{s+1} \al(1) \\
&= (\beta\al)^{s+2}(1). 
\end{align*}

That is, $(\beta\al)^{s+2}(1)=1$.
But $(\beta\al)^{k+1}=1$ and minimality of $k$ and $s$ imply 
%$k$ is the smallest number for which this happens. Hence 
$s=k-1$. Thus, the length of the cycle containing $\al(1)$ is $k+1$, which is the same as the length of the cycle containing $1$.
Since the same argument holds for any $i \in \mathbf{2n}$ there will be an even number of cycles of any given length in $\beta\al$.
\qed
\end{proof}

\begin{lemma}\label{lemma:withfp}
Let $\al$ and $\beta$ be permutations on the set $\mathbf{2n}$ such that $\al$ and $\beta$ are involutions. 
Then a cycle in $\beta \al$ contains at most $2$ points from $F_{\al} \cup F_{\beta}$.
\end{lemma}

\begin{proof}
If $F_{\al} \cup F_{\beta} = \varnothing$, then any cycle on $\beta\al$ contains $0$ elements of $F_{\al} \cup F_{\beta}$ and the statement is vacuously true.

Suppose then that $F_{\al} \cup F_{\beta} \neq \varnothing$.
Suppose $1 \in F_{\al}$; that is, $\al(1)=1$. If $1 \in F_{\beta}$, a similar argument would apply.

The cycle containing $1$ in $\beta \al$ is of the form
\[
\left( 1,\beta \al(1),\beta \al\beta\al(1),\hdots, (\beta \al)^k(1) \right)
\]
where $k$ is the smallest positive integer for which $(\beta \al)^{k+1}(1)=1$.  As in the proof of Lemma~\ref{lemma:fpfree}, this cycle contains $\al(1)=1=(\beta\al)^{k+1}(1)$.  
We have argued in the proof of Lemma~\ref{lemma:fpfree} that if $k+1$ is odd then
\[
(\beta \al)^{k+1}(1)=1=\al(1) \implies (\beta\al)^{(k+2)/2}(1) \in F_{\beta}
\]
and if $k+1$ is even then
\[
(\beta\al)^{k+1}(1)=1=\al(1) \implies (\beta\al)^{(k+1)/2}(1) \in F_{\al}.
\]

That is, if $1 \in F_{\al}$ then the cycle containing $1$ contains at least one other point from $F_{\al} \cup F_{\beta}$, namely $(\beta\al)^{(k+2)/2}(1)$ if the length of the cycle is odd, and $(\beta\al)^{(k+1)/2}(1)$ if it is even. 

Suppose that this cycle contains another point $i \in F_{\al} \cup F_{\beta}.$
We claim that $i$ must be one of the points identified above. Since $i$ is in the cycle, for some positive integer $s$, $i=(\beta\al)^s(1)$. Let $s$ be the smallest such integer.
  
If $i \in F_{\al}$; that is, $\al((\beta\al)^s(1))=(\beta\al)^s(1)$.  We then have
\begin{align*}
1&=(\al\beta)^s\al(\beta\al)^s(1) && \text{ since }((\beta\al)^s)^{-1}=(\al\beta)^s \\
&=\al(\beta\al)^{2s}(1).&&
\end{align*}
But $\al(1)=1$, so acting on both sides by $\al$ we have that $(\beta\al)^{2s}(1)=1$.

If $i \in F_{\beta}$, so that $\beta((\beta\al)^s(1))=(\beta\al)^s(1)$, we obtain $(\beta\al)^{2s-1}(1)=1$.

But since $k+1$ is the minimal integer for which $(\beta\al)^{k+1}(1)=1$, and $s$ is also minimal, it follows  $2s=k+1$ or $2s-1=k+1$ according to whether $i \in F_{\alpha}$ or $i \in F_{\beta}$. That is,

\[
s=\begin{cases}
(k+1)/2 & \text{ if } i \in F_{\alpha}, \\
(k+2)/2 & \text{ if } i \in F_{\beta}. 
\end{cases}
\]
Hence $i$ is one of the points identified above unless the two points are the same which will happen if,
\[k+1=\frac{k+1}{2} \text{ or } k+1=\frac{k+2}{2}.\]

The first equation does not have any nonnegative solution. The only nonnegative integer satisfying the second condition is $k=0$. If $k$ is $0$ i.e., $(\beta \al)(1)=1$ then since $1 \in F_{\al}$, it follows that $\beta(1)=1$ and hence $1$ is a fixed point of $\beta$ as well. In this case, the cycle of $\al\beta$ containing $1$ will be of length $1$ and hence contains a single point from $F_{\al} \cup F_{\beta}$. 

Thus, a cycle of $\beta \al$ contains no fixed points if there are no fixed points in $\beta$ and $\al$. 

A cycle of length $1$ contains a point $i$ from $F_{\al} \cup F_{\beta}$ if $i$ is fixed in both $\al$ and $\beta$, that is if $i \in F_{\al} \cap F_{\beta}$. 

If $F_{\al} \cap F_{\beta} = \emptyset$, then a cycle of $\beta\al$ contains exactly two points from $F_{\al} \cup F_{\beta}$.   
\qed 
\end{proof}
\citet{Petersen2013} also investigate the nature of the product of involutions and prove similar results. Their results are stated in terms of the structure of an involution product graph.

\section{Determining the DCJ distance}
\label{sec:dcjDisNew}

\subsection{Subpermutations and the link to transposition distance}

We define a binary relation on $\mathbf{2n}$ which will allow us to separate out the different components of a pair of genomic permutations, each of which we will then be able to sort independently of the others.

\begin{definition}
Let $\pi_1$ and $\pi_2$ be genomic permutations on $n$ regions. That is, $\pi_1$ and $\pi_2$ are involutions on the set $\mathbf{2n}$. Define the binary relation $\sim$ on $\mathbf{2n}$ by setting
\[i \sim j \iff (\pi_2 \pi_1)^k(i)=j \text{ or } \pi_1(\pi_2 \pi_1)^k(i)=j \text{ for some } k \in \mathbb{Z}.\]
\end{definition} 

The cycles of $\pi_1 \pi_2$ and $\pi_2 \pi_1$ are the same as sets (they are inverses of each other in $S_{2n}$), hence the binary relation $\sim$ defined for the pair $\pi_1,\pi_2$ would be the same as defined for the pair $\pi_2,\pi_1$. Therefore, without any ambiguity $\sim$ can be defined for an unordered pair of genomic permutations.  

\begin{lemma}
$\sim$ is an equivalence relation on $\mathbf{2n}$.
\end{lemma}
\begin{proof}
It is easy to verify that $\sim$ is reflexive and symmetric. To establish that $\sim$ is transitive, note that since $i$ could be related to $j$ through either of the two relations $\pi_1(\pi_2 \pi_1)^k(i)=j$ or $(\pi_2 \pi_1)^k(i)=j$, there are four possible cases to be checked. For example, suppose that $\pi_1(\pi_2\pi_1)^p(i)=j$ and $(\pi_2\pi_1)^q(j)=k$. Noting that $(\pi_2\pi_1)^q\pi_1=\pi_1(\pi_2\pi_1)^{-q}$, since $\pi_i$ are involutions, we have 
\begin{align*}
k&=(\pi_2\pi_1)^q\pi_1(\pi_2\pi_1)^p(i) \\
&=\pi_1(\pi_2\pi_1)^{p-q}(i),
\end{align*}
and hence $i \sim k$. The other cases can be checked similarly.
\qed
\end{proof}

For any $i,j \in \mathbf{2n}$, if $i$ and $j$ are in the same cycle of $\pi_2\pi_1$ then $j=(\pi_2\pi_1)^s(i)$ for some $s$ and hence $i \sim j$. Also, $i \sim \pi_1(i)$ and $\pi_1(i)$ is related to all the elements in the cycle of $\pi_2\pi_1$ that contains $\pi_1(i)$. Therefore an equivalence class under $\sim$ will be the union of the cycles of $\pi_1\pi_2$ containing $i$ and $\pi_1(i)$. 

Observe that $i \sim \pi_1(i)$ and $i \sim \pi_2(i)$. Hence the partition of $\mathbf{2n}$ under $\sim$ will also partition the $2$-cycles of $\pi_1$ and $\pi_2$. In what follows we would like to talk about the sub-permutations of $\pi_1$ and $\pi_2$ thus induced. Formally, 

\begin{definition}
Let $\pi_1$ and $\pi_2$ be genomic permutations, and let $\C_1,\C_2, \hdots, \C_r\subseteq\mathbf{2n}$ be the equivalence classes under $\sim$ defined by $\pi_1,\pi_2$. For $1\le s\le r$ and $i=1,2$, the \emph{sub-permutation} $\pi_i^{(s)}$ of $\pi_i$ induced by $\sim$ is defined to be the restriction of $\pi_i$ to $\C_s$, that is,
\[\pi_i^{(s)}:=\pi_i \big|_{\C_s}. \]

\end{definition}

Intuitively, we have collected in a sub-permutation all the $2$-cycles that are relevant for sorting $\pionesub{s}$ into $\pitwosub{s}$.

An example will help illustrate these definitions. Let $\pi_1$ and $\pi_2$ be the following genomic permutations on $4$ regions:
\[\pi_1=(1,6)(2,3)(4,5)(7,8),\quad \pi_2=(1,2)(3,4)(5,6).\]
The partitions of the set $\{1,2,\hdots,8\}$ under $\sim$ are $\C_1=\{1,2,3,4,5,6\}, \ \C_2=\{7,8\}$.

The sub-permutations $\pionesub{1}$ and $\pionesub{2}$ are then

$\pionesub{1}=(1,6)(2,3)(4,5)$, $\pionesub{2}=(7,8)$.
Similarly, the sub-permutations $\pitwosub{1}$ and $\pitwosub{2}$ are
$\pitwosub{1}=(1,2)(3,4)(5,6)$, $\pitwosub{2}=(7)(8)$ where the cycles of length 1 are written for clarity.

As remarked earlier, an equivalence class $\C_s$ under $\sim$ contains precisely those elements of $\mathbf{2n}$ that are contained in the cycle of $\pi_1\pi_2$ containing $i$ and $\pi_1(i)$. Therefore as proved in Lemmas~\ref{lemma:fpfree} and~\ref{lemma:withfp}, the product of the sub-permutations $\pionesub{s}$ and $\pitwosub{s}$ is either a single cycle containing one or two points from $F_{\pi_1} \cup F_{\pi_2}$ or a product of (exactly) two disjoint cycles. 

Suppose the sub-permutations $\pionesub{s}$ and $\pitwosub{s}$ are distinct. If either $\pionesub{s}$ and $\pitwosub{s}$ are conjugate in $S_{2n}$, then they have the same cycle type. Hence either $\C_s$ contains no points from $F_{\pi_1} \cup F_{\pi_2}$, or it contains one point each from $F_{\pi_1}$ and $F_{\pi_2}$. In the first case, it follows from Lemma~\ref{lemma:fpfree} that $\C_s$ contains an even number of points. In the latter case, the cardinality of $\C_s$ is odd.

These observations are summarised in corollary~\ref{cor:pi2pi1}. 

\begin{corollary}\label{cor:pi2pi1}
Let $\pionesub{s}$ and $\pitwosub{s}$ be distinct sub-permutations induced by an equivalence class $\C_s$ such that $\pionesub{s}$ and $\pitwosub{s}$ are conjugate in $S_{2n}$. Let $i \in \C_s$. The product $\pitwosub{s}\pionesub{s}$ is given by
\[
\pitwosub{s}\pionesub{s}=
\begin{cases}
 \left(i,\pi_2\pi_1(i), \hdots , (\pi_2\pi_1)^u(i) \right)\left(\pi_1(i),\pi_1\pi_2\pi_1(i), \hdots , (\pi_1\pi_2)^{u-1}\pi_1(i)\right)
 & \text{ if } F_{\pi_1} \cup F_{\pi_2}= \emptyset, \\
 \left(i,\pi_2\pi_1(i), \hdots (\pi_2\pi_1)^{2u}(i)\right)
 &\text{ if } F_{\pi_1} \cup F_{\pi_2} \neq \emptyset, i \in F_{\pi_1}.
\end{cases}
\]
The sum of lengths of the cycles in the product is the cardinality of $\C_s$.
\end{corollary}

Continuing our example above, we see $\pionesub{1}\pitwosub{1}=(1,3,5)(2,6,4)$, $\pionesub{2}\pitwosub{2}=(7,8)$.

Observe that $\pionesub{1}\pitwosub{1}$ is a product of two disjoint cycles. $\pionesub{2}\pitwosub{2}$ contains two points from $F_{\pi_2}$ namely $7$ and $8$.

If a partition $\C_t$ consists of a single point say $i$ then $i$ is a fixed point of both $\pi_1$ and $\pi_2$, hence the DCJ distance between sub-permutations induced by $\C_t$ is $0$.

We will determine the DCJ distance between $\pi_1$ and $\pi_2$ by determining the DCJ distance between $\pionesub{s}$ and $\pitwosub{s}$ for $s \in \{1,2, \hdots, r\}$. 

\begin{definition}
For any permutation $\pi$, the transposition length of $\pi$ denoted by $\newl_t(\pi)$ is the minimal number of transpositions needed to express $\pi$.
\end{definition}

Since the $D_{ij}$ operation involves multipliying a permutation with transpositions, we are interested in how multiplication by a transposition affects the tranposition length of a permutation. In fact this effect is easily stated: multiplication by a transposition changes the transposition length of a permutation by $\pm 1$. 

 That is, 
\begin{eqnarray}
	\label{equn:transLen} 
	&\newl_t\left((i,j\right)\pi)=\newl_t(\pi) \pm 1,\nonumber \\
	&\newl_t\left(\pi(i,j)\right)=\newl_t(\pi) \pm 1.
\end{eqnarray}

This can be observed by noting first that a permutation can be expressed as a product of either an odd or an even number of transpositions, but not both. That is,  the  parity of the number of transpositions needed to write a permutation as a product is unique. 

Let the transposition length of a permutation $\pi$ be $r$. That is,
\[\pi=t_1 t_2 \hdots t_r,\]
where the $t_i$ are transpositions. 

Suppose $r$ is odd. Then the parity of $(i,j)\pi$ is even since $(i,j) t_1 t_2 \hdots t_r$ is one expression of the result as a product of transpositions, although it may not be minimal. The transposition length of $(i,j)\pi$ is also therefore even, and hence it is either $r+1$ or $r-1$.  A similar argument follows if $r$ is even.

In the remaining part of this section, we will show that the DCJ distance between $\pi_1$ and $\pi_2$ can be determined in terms of the transposition length of the permutation product $\pi_2\pi_1$. First of all, note that if $\pi_1=\pi_2$ then $\pi_2\pi_1=()$ where $()$ is the identity permutation, hence $\newl_t\left(\pi_2\pi_1\right)=0$.

We make the following claim regarding a lower bound on the DCJ distance between permutations $\pi_1$ and $\pi_2$.

\begin{lemma}
\label{thm:lowerBound}
Let $\pi_1$ and $\pi_2$ be genomic permutations.  Then
\[d_{DCJ}\left(\pi_1,\pi_2 \right) \geq \frac{\newl_t\left(\pi_1 \pi_2\right)}{2}.\]

\end{lemma}
\begin{proof}
A single DCJ operation $D_{ij}$ acts either by conjugation of $\pi_1$ by the transposition $(i,j)$, or by multiplication of $\pi_1$ by $(i,j)$. 

If $D_{ij}(\pi_1)=(i,j)\pi_1$, then $D_{ij}(\pi_1)\pi_2=(i,j)\pi_1\pi_2$. Hence by equation~\eqref{equn:transLen}
\[\newl_t\left(D_{ij}(\pi_1) \pi_2\right)=\newl_t \left(\pi_1 \pi_2 \right) \pm 1.\]

If $D_{ij}(\pi_1)=(i,j)\pi_1(i,j)$ then
\[D_{ij}(\pi_1)\pi_2=(i,j)(\pi_1(i),\pi_1(j))\pi_1\pi_2.\]

By applying equation~\ref{equn:transLen} twice, 
\[\newl_t\left(D_{ij}(\pi_1) \pi_2\right)=\newl_t\left(\pi_1 \pi_2 \right) \text{ or }\newl_t \left( \pi_1 \pi_2 \right) \pm 2.\] 

Thus a single DCJ operation on $\pi_1$ can reduce the transposition length of $\pi_1 \pi_2$ by at most $2$. Since $\newl_t(\pi_2\pi_1)=0$ when $\pi_1=\pi_2$, the DCJ distance between $\pi_1$ and $\pi_2$ must be at least $\newl_t(\pi_1 \pi_2)/2$. 
\qed
\end{proof}
\subsection{DCJ distance between conjugate sub-permutations}
Let the sub-permutations $\pionesub{s}$ and $\pitwosub{s}$ be conjugate in $S_{2n}$. That is, there exists a $g\in S_{2n}$ such that
\[g\pionesub{s} g^{-1}=\pitwosub{s}.\] 
By writing $g$ as a product of transpositions, we obtain a sequence of DCJ operations that transforms $\pionesub{s}$ to $\pitwosub{s}$, each of which is conjugation by a transposition. Let $d^c_{DCJ}\left(\pionesub{s},\pitwosub{s}\right)$ be the minimal number of DCJ conjugation operations needed to transform $\pionesub{s}$  into $\pitwosub{s}$. Then clearly
\[d_{DCJ}\left(\pionesub{s},\pitwosub{s}\right) \leq d^c_{DCJ}\left(\pionesub{s},\pitwosub{s}\right).\] 

\begin{theorem}
\label{thm:conjugationDistance}
Let $\pionesub{s}$ and $\pitwosub{s}$ be sub-permutations of genomic permutations $\pi_{1}$ and $\pi_{2}$ on $n$ regions such that $\pionesub{s}$ and $\pitwosub{s}$ are conjugate in $S_{2n}$. Then the conjugation distance $d^c_{DCJ}\left(\pionesub{s},\pitwosub{s}\right)$ is half the transposition length of $\pitwosub{s} \pionesub{s}$, that is,
\[d^c_{DCJ}\left(\pionesub{s},\pitwosub{s}\right)=\frac{1}{2}\newl_t \left( \pitwosub{s} \pionesub{s} \right).\]

\end{theorem}
\begin{proof}
We prove the claim by induction on $r:=d^c_{DCJ}\left(\pionesub{s},\pitwosub{s}\right)$.

Suppose first that $d^c_{DCJ}\left(\pionesub{s},\pitwosub{s}\right)=1$. Because $\pionesub{s}$ and $\pitwosub{s}$ are conjugate,  $\pitwosub{s}\pionesub{s}$ is an even permutation, so $\newl_t \left( \pitwosub{s} \pionesub{s} \right)$ is at least $2$.

Since $r=1$, there exists a transposition $(i,j) \in S_{2n}$ such that $(i,j)\pionesub{s}(i,j)=\pitwosub{s}$. Hence,

\[\pitwosub{s} \pionesub{s} = (i,j)\pionesub{s}(i,j)\pionesub{s}=(i,j)(\pionesub{s}(i),\pionesub{s}(j)).\]

Therefore the transposition length of $\pitwosub{s} \pionesub{s}$ is $2$. 

Assume that the hypothesis is true for all $r \in \mathbb{N}$, with $r < u $. That is, for $r<u$,
\[r=d^c_{DCJ}\left(\pionesub{s},\pitwosub{s}\right)=\frac{1}{2} \newl_t\left(\pitwosub{s}\pionesub{s}\right).\]

Next suppose that $d^c_{DCJ}\left(\pionesub{s},\pitwosub{s}\right)=u$. That is,

\[w_u w_{u-1}\hdots w_1 (\pionesub{s}) w_1 \hdots w_{u-1} w_{u}=\pitwosub{s},\]
for some  transpositions $w_i \in S_{2n}$ with $u$ minimal.

Let $\pi_1'=w_{u-1}\hdots w_1 (\pionesub{s}) w_1 \hdots w_{u-1}$. 

Since the conjugation distance between $\pi_1'$ and $\pi_1$ is $u-1 < u$, from the induction hypothesis we know that $\newl_t\left(\pi_1'\pionesub{s}\right)=2(u-1)$.

Write \[\pi_1'\pionesub{s}=t_1 t_2 \hdots t_{2(u-1)}\] for  transpositions $t_i \in S_{2n}$.

Since $w_u \pi_1' w_u=\pitwosub{s}$, we have $d^c_{DCJ}\left(\pi_1',\pitwosub{s}\right)=1$, and hence $\newl_t\left(\pitwosub{s}\pi_1' \right)=2$. That is,
$\pitwosub{s}\pi_1'=v_1v_2'$ where  $v_1$ and $v_2$ are transpositions in $S_{2n}$.

Then 
\[\pitwosub{s}\pionesub{s}=\pitwosub{s}\pi_1'\pi_1'\pionesub{s}=v_1v_2t_1 t_2 \hdots t_{2(u-1)}.\]

This is a product of a permutation of transposition length $2(u-1)$ with two transpositions. The transposition length of the result will be $\newl_t\left(\pitwosub{s}\pionesub{s}\right) \in \{2u-4,2u-2,2u\}$.  However, if $\newl_t \left(\pitwosub{s}\pionesub{s} \right)<2u $ then $d^c_{DCJ}(\pionesub{s},\pitwosub{s}) < u$, contrary to our assumption. Hence $\newl_t(\pitwosub{s}\pionesub{s})=2u$. That is, 
\[d^c_{DCJ}\left(\pionesub{s},\pitwosub{s}\right)=u \implies \newl_t \left(\pitwosub{s}\pionesub{s}\right)=2u.\]
\qed
\end{proof}

Putting together the lower bound for DCJ distance from Lemma~\ref{thm:lowerBound} with the upper bound from Theorem~\ref{thm:conjugationDistance}, we have the following.

\begin{corollary}
If $\pionesub{s}$ and $\pitwosub{s}$ are conjugate sub-permutations of the genomic permutations $\pi_{1}$ and $\pi_{2}$ on $n$ regions, 
then 
\[d_{DCJ}\left(\pionesub{s},\pitwosub{s}\right)=\frac{\newl_t\left(\pitwosub{s}\pionesub{s}\right)}{2}.\]
\end{corollary}

\subsection{Constructing a sorting element for conjugate sub-permutations}

When the sub-permutations $\pionesub{s}$ and $\pitwosub{s}$ induced by $\C_s$ are conjugate, an element of minimal length sorting $\pionesub{s}$ into $\pitwosub{s}$ can easily be constructed as follows.  Corollary~\ref{cor:pi2pi1} gives the structure of the product $\pitwosub{s}\pionesub{s}$.

If $\pionesub{s}$ and $\pitwosub{s}$ have no fixed points then 
\begin{equation}\label{eq:double.cycle}
\pitwosub{s}\pionesub{s}=\left(1,\pi_2\pi_1(1),(\pi_2\pi_1)^2(1), \hdots, (\pi_2\pi_1)^u(1)\right)\left(\pi_1(1), \pi_2(1),\pi_2\pi_1\pi_2(1),\hdots, ( \pi_2\pi_1)^{u-1}\pi_2(1)\right).
\end{equation}

If $\pionesub{s}$ and $\pitwosub{s}$ contain fixed points then $\pitwosub{s} \pionesub{s}$ is a  single cycle 
\begin{equation}\label{eq:single.cycle}
\pitwosub{s}\pionesub{s}=\left(1,\pi_2\pi_1(1),(\pi_2\pi_1)^2(1), \hdots, (\pi_2\pi_1)^{2u}(1)\right). 
\end{equation}

Let $g=\left(1,\pi_2\pi_1(1),(\pi_2\pi_1)^2(1) \hdots ,(\pi_2\pi_1)^u(1)\right)$. We claim that 
\[g\pionesub{s}g^{-1}=\pitwosub{s}.\]

If $i$ is moved by $g$, then $g(i)=\pi_2\pi_1(i)$. For any $2$-cycle $\left(i,\pi_1(i)\right)$ in $\pionesub{s}$, either $i$ or $\pi_1(i)$ is moved by $g$. If $\pitwosub{s}\pionesub{s}$ is a product of two cycles then as proved in Lemma~\ref{lemma:withfp}, $i$ and $\pi_1(i)$ are in different cycles of $\pitwosub{s}\pionesub{s}$ . Since $g$ is precisely one of the two cycles of $\pitwosub{s}\pionesub{s}$, it moves exactly one of $i$ and $\pi_1(i)$. 

On the other hand, if $\pitwosub{s}\pionesub{s}$ is a single cycle as in Eq~\eqref{eq:single.cycle}, then suppose $\pi_1(1)=1$ (if instead $\pi_2(1)=1$, a similar argument would apply). Then,
\begin{align*}
1 &=\left(\pitwosub{s}\pionesub{s}\right)^{2u+1}(1) \\
  &=\pitwosub{s}\left(\pionesub{s}\pitwosub{s}\right)^{2u}(1).
\end{align*}
This implies that  $\left(\pionesub{s}\pitwosub{s}\right)^{2u}(1)=\pitwosub{s}(1)$. Now suppose $i=\left(\pitwosub{s}\pionesub{s}\right)^b(1)$ for some $b$ then $\pionesub{s}(i)=\left(\pitwosub{s}\pionesub{s}\right)^a(i)$ for some $a$, since $i$ and $\pionesub{s}(i)$ are in the same cycle. Also,
\begin{align*}
\pionesub{s}(i)&=\pionesub{s}\left(\pitwosub{s}\pionesub{s}\right)^b(1)\\
&=\left(\pionesub{s}\pitwosub{s}\right)^b(1) \\
&=\left(\pionesub{s}\pitwosub{s}\right)^{-2u+b}\left(\pionesub{s}\pitwosub{s}\right)^{2u}(1) \\
&=\left(\pionesub{s}\pitwosub{s}\right)^{-2u+b}\pitwosub{s}(1)=\left(\pionesub{s}\pitwosub{s}\right)^{-2u+b-1}\\
&=\left(\pitwosub{s}\pionesub{s}\right)^{2u-b+1}(1).
\end{align*}

If $i$ is moved by $g$, that is if $i$ is in the cycle $\left(1,\pi_2\pi_1(1),\hdots,(\pi_2\pi_1)^u(1)\right)$,  then $b \leq u $ which means that $a=2u-b+1 > u$ and $\pionesub{s}(i)$ is not in this cycle and hence not moved by $g$.

If $\pionesub{s}(i)$ is moved by $g$, that is if $\pionesub{s}(i)$ is in the cycle $\left(1,\pi_2\pi_1(1),\hdots,(\pi_2\pi_1)^u(1)\right)$,  then $a= 2u-b+1 \leq u $ which means that $u+1 \leq b$ and $i$ is not in this cycle and hence not moved by $g$.

Thus we have established that whether the product $\pitwosub{s}\pionesub{s}$ is given by Eq~\eqref{eq:double.cycle} or Eq~\eqref{eq:single.cycle}, for any $i \in \C_s$, $g$ moves either $i$ or $\pionesub{s}(i)$.

So if $i$ is moved by $g$, then $g(i)=\pi_2\pi_1(i)$, and since $\pi_1(i)$ is then not moved by $g$, $g(\pi_1(i))=\pi_1(i)$.
Consider the product $g \pionesub{s} g^{-1}$. For any $2$-cycle $(i,\pi_1(i))$ in $\pionesub{s}$ suppose $i$ is moved by $g$. We have
 
\[g\left(i,\pi_1(i)\right)g^{-1}=\left(g(i),g\left(\pi_1(i)\right)\right)=\left(\pi_2\left(\pi_1(i)\right),\pi_1(i)\right)\]
which is the $2$-cycle in $\pitwosub{s}$ containing $\pi_1(i)$. Thus

\[g \pionesub{s} g^{-1}=\pitwosub{s}.\]

Since the transposition length of $g$ is $u$ (that is, we require $u$ transpositions to express $g$ as a product), and conjugation by a $2$-cycle $(i,j)$ is one $D_{ij}$ operation, we require $u$ DCJ operations to sort $\pionesub{s}$ into $\pitwosub{s}$. This is exactly the DCJ distance between them. Hence $g$ gives an optimal sorting element that is, $g \pionesub{s} g^{-1}=\pitwosub{s}$.
 
As there is nothing special about the choice of $1$ in this argument, the cycle 
\[\left(\pi_1(1), \pi_2(1),\pi_2\pi_1\pi_2(1),\hdots ,( \pi_2\pi_1)^{u-1}\pi_2(1)\right)\]
is also an optimal sorting element.

The above construction might be better understood through an example. Let $\pi_1$ and $\pi_2$ be the following genomic permutations on $6$ regions:
\[\pi_1=(1,6)(2,3)(4,5) ,\quad \pi_2=(1,2)(3,4)(5,6).\]

Since $\pi_1$ and $\pi_2$ have the same cycle structure, we know that they are conjugate in $S_{2n}$. Consider the product $\pi_1\pi_2$.
\[\pi_1\pi_2=(1,6)(2,3)(4,5)(1,2)(3,4)(5,6)=(1,3,5)(2,6,4).\]

Let $g=(1,3,5)$. Then $g^{-1}=(1,5,3)$.

\[g\pi_2g^{-1}=(1,5,3)(1,2)(3,4)(5,6)(1,3,5)=(1,6)(2,3)(4,5)=\pi_1.\]

The above discussion is summarised in Lemma~\ref{lemma:scenario}.
\begin{lemma} \label{lemma:scenario}
Let $\pionesub{s},\pitwosub{s}$ and $u$ be as in Corollary~\ref{cor:pi2pi1}.  An  element $g$ such that $\newl_t(g)=d_{DCJ}(\pionesub{s},\pitwosub{s})$ that sorts $\pionesub{s}$ into $\pitwosub{s}$ can be constructed from the product $\pitwosub{s}\pionesub{s}$ as 
\[g=\left( 1,\pi_2\pi_1(1), (\pi_2\pi_1)^2(1),\hdots , (\pi_2\pi_1)^u(1) \right).\]
\end{lemma}

A similar construction has been given in \cite{Feijao2013}, where they construct the sorting element by establishing a correspondence between the connected components of adjacency graph and the permutation product.
%%%%%%%%%%%%%%%%%%%%%%%%%%%%%%%%%%%%%%%%%%%%%%%%%%%%%%%%%%%%%%%%%%%%%%%%%%%%%%%%%%%%%% 
\subsection{DCJ distance between non-conjugate sub-permutations}

We will now consider the case where $\pionesub{s}$ and $\pitwosub{s}$ are not conjugate in $S_{2n}$.  

\begin{theorem}
\label{thm:nonConjugateDis}
Let $\pionesub{s}$ and $\pitwosub{s}$ be non-conjugate sub-permutations of the genomic permutations $\pi_{1}$ and $\pi_{2}$ on $n$ regions. % such that $\pionesub{s}$ and $\pitwosub{s}$ are not conjugate in $S_{2n}$. 
Then
\[d_{DCJ}\left(\pionesub{s},\pitwosub{s}\right)=\frac{1}{2}\left(\newl_t(\pitwosub{s}\pionesub{s})+1\right).\]
\end{theorem}
\begin{proof}
We have remarked earlier that for any $i$ in the equivalence class $\C_s$, $\C_s$ contains only the elements contained in the cycles of $\pi_1\pi_2$containing $i$ and $\pi_1(i)$. Hence as proved in Lemma~\ref{lemma:withfp} if the induced sub-permutations $\pitwosub{s}\pionesub{s}$ are not identical, their product contains exactly two fixed points. Since $\pionesub{s}$ and $\pitwosub{s}$ are not conjugate, both the fixed points belong to the same sub-permutation.

Suppose the fixed points are $i_1$ and $i_2$ and that they belong to $\pionesub{s}$. Recall that if $i,j$ are both fixed in $\pi$ then $D_{ij}(\pi)=(i,j)\pi$, and hence
\[
D_{i_1i_2}\left(\pionesub{s}\right)=(i_1,i_2)\pionesub{s}.
\]
While $\pionesub{s}$ and $\pitwosub{s}$ are both products of $2$-cycles, the number of $2$-cycles in $\pionesub{s}$ is one less than the number of $2$-cycles of $\pitwosub{s}$, since it has two fixed points. Therefore $D_{i_1i_2}(\pionesub{s})$ is conjugate to $\pitwosub{s}$. Similarly, if the two fixed points belong to $\pi_2$ then $D_{i_1i_2}(\pitwosub{s})$ is conjugate to $\pionesub{s}$. As the DCJ distance is symmetric we can assume without loss of generality that the fixed points belong to $\pionesub{s}$.

Let $\pi'=D_{i_1i_2}(\pionesub{s})$. From Theorem~\ref{thm:conjugationDistance} we have that
\[
d_{DCJ}\left(\pi',\pitwosub{s}\right)=\frac{1}{2}\newl_t\left(\pitwosub{s} \pi'\right) = \frac{1}{2}\newl_t\left(\pi' \pitwosub{s}\right)=\frac{1}{2}\newl_t\left((i_1,i_2)\pionesub{s}\pitwosub{s}\right).
\]

Since $i_1$ and $i_2$ are in the same cycle of $\pionesub{s}\pitwosub{s}$, multiplication by $(i_1,i_2)$ will split this cycle into two cycles, reducing the transposition length of the product by 1. Hence 
\[
d_{DCJ}\left(\pi',\pitwosub{s}\right)=\frac{1}{2}\newl_t\left((i_1,i_2)\pionesub{s}\pitwosub{s}\right)=\frac{1}{2}\left(\newl_t\left(\pionesub{s}\pitwosub{s}\right)-1\right).\]

The DCJ distance between $\pionesub{s}$ and $\pi'$ is 1. Thus the triangle inequality gives
\begin{align*}
d_{DCJ}\left(\pionesub{s},\pitwosub{s}\right) &\leq d_{DCJ}\left(\pionesub{s},\pi'\right)+d_{DCJ}\left(\pi',\pitwosub{s}\right)\\
&=\frac{1}{2}\left(\newl_t\left(\pionesub{s}\pitwosub{s}\right)-1\right)+1\\
&=\frac{1}{2}\left(\newl_t\left(\pionesub{s}\pitwosub{s}\right)+1\right).
\end{align*}

At the same time we have a lower bound on the distance (Lemma~\ref{thm:lowerBound}), so that
\[ \frac{1}{2}\newl_t\left(\pionesub{s}\pitwosub{s}\right) \leq d_{DCJ}\left(\pionesub{s},\pitwosub{s}\right) \leq \frac{1}{2}\left(\newl_t\left(\pionesub{s}\pitwosub{s}\right)+1\right).\]

Since the DCJ distance is an integer (the number of DCJ operations), and the transposition length of $\pionesub{s}\pitwosub{s}$ is odd, we have
\[
d_{DCJ}\left(\pionesub{s},\pitwosub{s}\right)=\frac{1}{2}\left(\newl_t\left(\pionesub{s}\pitwosub{s}\right)+1\right)
\]
as required.
\qed
\end{proof}

From Theorems~\ref{thm:conjugationDistance} and~\ref{thm:nonConjugateDis}, it is clear that the sub-permutations induced by the partitions $\C_1,\C_2, \hdots, \C_r$ under $\sim$ can be sorted independently. Therefore
\[d_{DCJ}(\pi_1,\pi_2) \leq \sum_s{d_{DCJ}\left(\pionesub{s},\pitwosub{s}\right)}.\]

We claim that a sorting sequence that involves a DCJ operation $D_{ij}$, where $i,j$ are in different partitions $\C_s$, cannot be shorter than a sequence that sorts each partition independently.  

Let $\C_r$ and $\C_s$ be distinct equivalence classes of $\mathbf{2n}$ under $\sim$. Let $\newl_t\left(\pionesub{r}\pitwosub{r}\right)$ be $l_1$ and $\newl_t\left(\pionesub{s}\pitwosub{s}\right)$ be $l_2$.  Then
\[
\frac{1}{2}\left(l_1\right) \leq d_{DCJ}\left(\pionesub{r},\pitwosub{r}\right) \leq \frac{1}{2}\left(l_1+1\right).
\]

Similarly 
\[
\frac{1}{2}(l_2) \leq d_{DCJ}\left(\pionesub{s},\pitwosub{s}\right) \leq \frac{1}{2}\left(l_2+1\right).
\]
Combining these, we have
\[\frac{1}{2}\left(l_1+l_2\right) \leq d_{DCJ}(\C_r) + d_{DCJ}(\C_s) \leq \frac{1}{2}\left(l_1+l_2+2\right).\]

The action of $D_{ij}$ on $\pi_1$ may combine the two partitions $\C_r$ and $\C_s$ into a single partition $\C_{t}$ or it may change each of the partitions $\C_r$ and $\C_s$. In the latter case, by an abuse of notation we use $\C_{t}$ to denote the union of the partitions changed by the action of $D_{ij}$. We wish to determine the transposition length of $\pionesub{t}\pitwosub{t}$ in order to find the DCJ distance. 

The action of $D_{ij}$ on $\pi_1$ may be left multiplication by $(i,j)$ or conjugation by $(i,j)$.  If it acts by left multiplicaton, so that $D_{ij}(\pi_1)=(i,j)\pi_1$, then since $i$ and $j$ are in different partitions and hence in different cycles of $\pionesub{r}\pitwosub{r}\pionesub{s}\pitwosub{s}$, multiplication by $(i,j)$ will combine the two cycles that contain $i$ and $j$. The transposition length of the product will therefore increase by $1$. That is, in this case,
\[\newl_t\left(\pionesub{t}\pitwosub{t}\right)=\newl_t\left(\pionesub{r}\pitwosub{r}\pionesub{s}\pitwosub{s}\right)+1.\] 

On the other hand, if $D_{ij}$ acts by conjugation then $D_{ij}(\pi_1)=(i,j)\pi_1(i,j)$, and we have that
	\[(i,j)\pi_1(i,j)=(i,j)\left(\pi_1(i),\pi_1(j)\right)\pi_1.\]

The images $\pi_1(i)$ and $\pi_1(j)$ are in different cycles of $\pionesub{r}\pitwosub{r}\pionesub{s}\pitwosub{s}$ since $i$ and $j$ are in different partitions. In $x=\left(\pi_1(i),\pi_1(j)\right)\pi_1\pi_2$, the cycles containing $\pi_1(i)$ and $\pi_1(j)$ will combine into a single cycle, increasing the length of the product by $1$. Then the cycle of $x$ containing $\pi_1(i)$ and $\pi_1(j)$ with contain either both, one, or neither of  $i$ and $j$. 

Accordingly, multiplication by $(i,j)$ will either split this cycle into two cycles or combine two different cycles into one cycle. Thus the transposition length may increase or decrease by 1 (from the previous step). Hence

\[\newl_t\left(\pionesub{t}\pitwosub{t}\right)=\newl_t\left(\pionesub{r}\pitwosub{r}\pionesub{s}\pitwosub{s}\right),\]
or
\[\newl_t\left(\pionesub{t}\pitwosub{t}\right)=\newl_t\left(\pionesub{r}\pitwosub{r}\pionesub{s}\pitwosub{s}\right)+2.\] 
In both cases $l_1+l_2 \leq \newl_t(\C_{t})$ and hence 
\[ \frac{1}{2}\left(l_1+l_2\right) \leq d_{DCJ}\left(\pionesub{t},\pitwosub{t}\right).\] 

Since one DCJ operation was needed to change the partition $\C_r$ and $\C_s$ into $\C_{t}$, and the DCJ distance of the sub-permutations induced by $\C_{t}$ is at least $\frac{1}{2}(l_1+l_2)$, any sorting scenario for $\C_r$ and $\C_s$ that steps through $\C_{t}$ is of length at least $\frac{1}{2}(l_1+l_2)+1$. At the same time, the sum of the distances of the partitions $\C_r$ and $\C_s$ is bounded above by:

\[d_{\C_r}+d_{\C_s} \leq \frac{1}{2}(l_1+l_2)+1. \]

Therefore we conclude that no sequence of DCJ operations sorting $\pi_1$ into $\pi_2$ can be shorter than a sequence that sorts the sub-permutations independently. 
\begin{theorem}
\label{thm:dcjDisTotal}
Let $\pi_1$ and $\pi_2$ be genomic permutations on $n$ regions. The DCJ distance between $\pi_1$ and $\pi_2$ is given by
\[d_{DCJ}(\pi_1,\pi_2)=\frac{1}{2}\left(\newl_t(\pi_2\pi_1)+ n_c\right)\]
where $n_c$ is the number of cycles in the product $\pi_2\pi_1$ which contain two fixed points of $F_{\pi_1}$ or $F_{\pi_2}$.
\end{theorem}

%=============================================%
\section{Counting the optimal sorting scenarios}\label{sec:numb.scenarios}
%=============================================%
To sort a permutation $\pi_a$ into $\pi_b$ means to transform $\pi_a$ into $\pi_b$ through a sequence of allowed operations (in this case the DCJ operation). 
A sorting scenario is defined as follows.

\begin{definition} %{Sorting scenario}
A \emph{sorting scenario} of length $k$ that sorts genomic permutation $\pi_a$ into genomic permutation $\pi_b$ is a sequence of genomic permutations 
\[ \{(\pi_a=)\pi_0,\pi_1,\pi_2, \hdots,  \pi_{k-1} ,\pi_k(=\pi_b)\} ,\]
such that each element of the sequence is obtained from the previous element through a single DCJ operation.
\end{definition}

That is, a sorting scenario is the  sequence of permutations we step through as $\pi_a$ is sorted into $\pi_b$ through DCJ operations.  If the DCJ distance is $d$, then an optimal sorting scenario (scenario of minimal length) is a sequence of length $d+1$. Two optimal sorting scenarios are equal if they are equal as sequences i.e., corresponding terms are equal.

The total number of optimal sorting scenarios between a pair of genomes is an interesting and important question. In constructing a phylogenetic history, the minimal distance with respect to some mutational operation is used. However such a minimal path is seldom unique. Hence as \citet{miklos2009efficient} and \citet{siepel2002algorithm} point out, it would be more appropriate to account for and average over  all possible evolutionary paths to draw meaningful statistical inferences. 

\citet{Braga2010} extended their earlier work \citep{Braga2009} to give a closed formula for the number of optimal DCJ sorting scenarios for certain instances of the problem. \citet{ouangraoua2010} also present similar results for the number of optimal sorting scenarios and establish connections between the number of sorting scenarios and other combinatorial objects such as parking functions. 

Considering genomes as permutations and the DCJ as an action on a permutation allows us to count the number of optimal sorting scenarios for a subset of genomes in a straightforward manner. The subset of genomes we consider are those that are conjugate in $S_{2n}$. Our result is equivalent to the results obtained by the previous papers.

Let $\pi$ be a genomic permutation on $n$ regions and let $i,j \in \mathbf{2n}$.  The restriction of $\pi$ to the cycles containing $i$ and $j$ is $\left(i,\pi(i)\right)\left(j,\pi(j)\right)$. As $D_{ij}$ acting on $\pi$ only affects the cycles containing $i$ and $j$,
\begin{align*}
D_{ij}(\pi)
&=(i,j)\left(i,\pi(i)\right)\left(j,\pi(j)\right)(i,j) \\
% &=(i,j)\left(i,\pi(i)\right)\left(j,\pi(j)\right)(i,j) \\
&=\left(i,\pi(j)\right)\left(j,\pi(i)\right) \\
&=\left(\pi(i),\pi(j)\right)\left(i,\pi(i)\right)\left(j,\pi(j)\right)\left(\pi(i),\pi(j)\right) \\
&=D_{\pi(i)\pi(j)}(\pi).
\end{align*}

If the restriction of $\pi$ to the cycles containing $i$ and $j$ is $\left(i\right)\left(j,\pi(j)\right)$, then
\[D_{ij}(\pi)=(i,j)\left(i\right)\left(j,\pi(j)\right)(i,j)=\left(j\right)\left(i,\pi(j)\right).\]

In this case there is no $D_{kl}$ such that $D_{kl}(\pi)=D_{ij}(\pi)$. As these are the two cases where the algebraic DCJ operator acts via conjugation (see Definition~\ref{d:Dij}), we have the following lemma.

\begin{lemma}\label{lemma:samess}
Let $\pi$ be a genomic permutation on $n$ regions. Let $D_{ij}$ and $D_{kl}$ act on $\pi$ via conjugation. If
\[D_{ij}\left(\pi\right)=D_{kl}\left(\pi\right),\]
then either $(i,j)=(k,l)$ or $(i,j)$ and $(k,l)$ are disjoint transpositions such that $(k,l)=\left(\pi(i),\pi(j)\right)$.
\end{lemma}

Based on the characterization of the DCJ operators that act in the same way on a genomic permutation, we can easily enumerate the sorting scenarios.

\begin{theorem}
\label{thm:sortingScen}
Let $\pi_a$ and $\pi_b$ be genomic permutations on $n$ regions such that $\pi_a$ and $\pi_b$ are conjugate in $S_{2n}$. If the DCJ distance $d_{DCJ}\left( \pi_a,\pi_b\right)=d$ then the number of optimal DCJ sorting scenarios sorting $\pi_a$ into $\pi_b$ is $(d+1)^{d-1}$.
\end{theorem}
\begin{proof}

As we have seen in the proof of Theorem~\ref{thm:conjugationDistance}, if the DCJ distance is $d$then we can construct an element  $g \in S_{2n}$ such that $g$ is a cycle of length $d+1$ (and consequently of transposition length $d$) and \[g\pi_ag^{-1}=\pi_b.\] 
Lemma~\ref{lemma:scenario} gives the construction of a cycle $g$ such that $g\pi_ag^{-1}=\pi_b$.  Let $g$ be as in the statement of Lemma~\ref{lemma:scenario} i.e.,
$g=\left(1, \pi_b\pi_a(1), \left(\pi_b\pi_a\right)^2(1), \hdots , \left(\pi_b\pi_a\right)^d(1)\right)$.
The number of ways to represent a cycle of length $n$ as the product of $n-1$ transpositions (i.e., as a minimal product) is $n^{n-2}$~\citep{denes1959representation}. 
Hence $g$ can be written as a product of transpositions in $(d+1)^{d-1}$ ways.  It remains to show that 
\begin{enumerate}
	\item each expression of $g$ as a minimal product of transpositions corresponds to a distinct sorting scenario, and
	\item there can be no other sorting scenarios. That is, if there is $h \in S_{2n}$ such that $\newl_t(h)=\newl_t(g)$ and $h \pi_a h^{-1}$, then any sorting scenario produced by $h$ is identical to some sorting scenario produced by $g$.
\end{enumerate}

Let $S(g)$ be the set of all expressions for $g$ as a minimal product of transpositions. For example, if $g=(1,3,5)$ then $S(g)=\{(1,5)(1,3), \enskip (1,3)(3,5), \enskip (3,5)(1,5)\}$.

\emph {1. Claim: each expression of $g$ as a minimal product of transpositions corresponds to a distinct sorting scenario}.

From the construction of $g$ preceding Lemma~\ref{lemma:scenario}, we know that for any $i \in \mathbf{2n}$, $g$ moves either $i$ or  $\pi_a(i)$ but not both. Suppose $g$ moves $i$. Then since $g$ fixes $\pi_a(i)$, in a minimal factorization of $g$ as a product of transpositions, no transposition moves $\pi_a(i)$. 

To observe this, note that the number of trees on $d$ labeled vertices is $(d+1)^{d-1}$ (given by Cayley's formula). Thus there is a bijection between the $S(g)$ and the set of trees on $d$ vertices. A minimal factorization of $g$ into transpositions can be associated with a tree by considering a transposition $(i,j)$ to correspond to the edge between vertices $i$ and $j$. If a point $\pi_a(i)$ fixed by $g$ is moved by some transposition in a minimal factorization of $g$, then the factorization must contain a cycle that would move $\pi_a(i)$ back to itself. But such a cycle would correspond to a loop in the graph corresponding to the factorization, which cannot be as the graph is a tree. Hence in a minimal factorization of $g$ as a product of transpositions, no transposition moves $\pi_a(i)$.

Suppose $u_du_{d-1}\hdots u_1$ and $w_d w_{d-1}\hdots w_1$ are distinct elements of $S(g)$ that produce the same sorting scenarios.  We will derive a contradiction.
 
Let $k$ be the lowest index such that $u_k \neq w_k$. Let $w_{k-1}\hdots w_1=u_{k-1}\hdots u_1 =g^{\prime}$ and let $g^{\prime} \pi_a g^{\prime-1}=\pi_{k-1}$. 
 
 Let $u_k=(i,j)$. By Lemma~\ref{lemma:samess}, $u_k$ and $w_k$ are disjoint and $w_k=\left(\pi_{k-1}(i),\pi_{k-1}(j)\right)$.
 
Now, $u_k$ is a transposition in the minimal expression for $g$. Since $u_k$ moves $i$,  $i$ is in the support of $g$, and $\pi_a(i)$ is not in the support of $g$. The support of $g^{\prime}$ is a subset of the support of $g$, hence $\pi_a(i)$ is not in the support of $g^{\prime}$.
 
 If $u_k$ is disjoint from $g^{\prime}$, then 
 
 \[\pi_{k-1}(i)=g^{\prime} \pi_a g^{\prime-1}(i) =\pi_a(i).\]
 
 Hence $\pi_{k-1}(i)$ is not in the support of $g$.

 If, on the other hand, $u_k$ is not disjoint from $g^{\prime}$, then let $i$ be in the support of $g^{\prime}$.
 
Clearly, $g^{\prime-1}(i)$ is in the support of $g^{\prime}$ (it gets mapped to $i$ by $g^{\prime}$)  and hence in the support of $g$. Therefore, $\pi_a\left((g^{\prime}) ^{-1}(i)\right)$ is not in the support of $g$ (and hence not in the support of $g^{\prime}$)  since for any $i \in \mathbf{2n}$, $g$ moves either $i$ or $\pi_a(i)$. Hence
 \[\pi_{k-1}(i)=g^{\prime} \pi_a g^{\prime-1}(i)=\pi_a\left(g^{\prime-1} (i)\right).\]
 
That is, $\pi_{k-1}(i)$ is not in the support of $g$. 
 
Thus, in both cases (whether $u_k$ is disjoint from $g^{\prime}$ or not), $w_k$ moves an element that is not in the support of $g$. This contradicts the assertion that $w_d w_{d-1} \hdots w_1=g$.  Thus either $w_d w_{d-1}\hdots w_1 \notin S(g)$ or the sorting scenarios produced by distinct elements are distinct.
Each expression of $g$ as a minimal product of transpositions therefore gives a unique sorting scenario and the number of sorting scenarios is at least the cardinality of $S(g)$.

\emph{2. Claim: there are no additional sorting scenarios}.

Let $h\in S_{2n}$ such that $\newl_t(h)=\newl_t(g)=d$ and  $h \pi_a h^{-1}=\pi_b$. Let $h=w_dw_{d-1}\hdots w_1$ be a factorization of $h$ into transpositions. We claim that $h$ produces the same sorting scenario as some element in $S(g)$. This will establish that the number of sorting scenarios is equal to the cardinality of $S(g)$. To prove this assertion, we first prove that there is some element $u_d \hdots u_1 \in S(g)$ such that $u_1 \pi_a u_1=w_1 \pi_a w_1$.   

Suppose that this is not the case. That is, no element in $S(g)$ produces a sorting scenario that has  $w_1 \pi_a w_1$ as the second term (the first term in all sorting scenarios is $\pi_a$).  Consider the element
\[h^{\prime}=u_d u_{d-1} \hdots u_1 w_1.\]
Let $w_1 \pi_a w_1 = \pi_a^{\prime}$. Then

\begin{align*}
h^{\prime} \pi_a^{\prime} h^{\prime-1}&= \left(u_d u_{d-1} \hdots u_1 w_1\right)\left(w_1 \pi_a w_1\right)\left(w_1u_1 \hdots w_{d-1}w_d\right)\\
&= \left(u_d u_{d-1} \hdots u_1\right)\left(\pi_a\right)\left(u_1 \hdots w_{d-1}w_d\right)\\
&=\pi_b.
\end{align*}
At the same time, the DCJ distance between $\pi_a^{\prime}$ and $\pi_b$ is $d-1$, since $(w_d \hdots w_2) \pi_a^{\prime} (w_2 \hdots w_d)=\pi_b$. Hence $u_d u_{d-1} \hdots u_1 w_1$ can be written as a product of $d-1$ transpositions say $v_{d-1} \hdots v_1$. Then $ v_{d-1} \hdots v_1w_1=g$, and is an expression of length $d$ equal to $g$ that is not is $S(g)$ because we have assumed that there is no element in $S(g)$ such that $u_1 \pi_a u_1 = w_1 \pi_a w_1$. 
This is a contradiction since $S(g)$ by definition contains all factorizations of $g$ of length $d$.   Thus there is some element in $u_d \hdots u_1 \in S(g)$ such that $w_1 \pi_a w_1 = u_1 \pi_a u_1$.

Let $S_1(g)=\{u_d \hdots u_1 \in S(g) \mid u_1 \pi_a u_1 =w_1 \pi_a w_1\}$.

By a similar argument we can prove that there exists some element in $S_1(g)$ such that $u_2 (u_1 \pi_a u_1) u_2 = w_2 (u_1\pi_au_1) w_2$.  In general, let
\[S_k(g)=\{u_d \hdots u_1 \in S_{k-1}(g) \mid u_k(u_{k-1} \hdots u_1\pi_a u_1 \hdots u_{k-1})u_k=w_k(w_{k-1} \hdots w_1\pi_a w_1 \hdots w_{k-1})w_k \}.\]

Suppose that there does not exist any element in $S_k(g)$ such that 
\[u_{k+1}(u_{k} \hdots u_1\pi_a u_1 \hdots u_k)u_{k+1}=w_{k+1}(u_k \hdots u_1 \pi_a u_1 \hdots u_k).\]
Let $u_d \hdots u_1 \in S_k(g)$ and let $u_{k} \hdots u_1\pi_a u_1 \hdots u_k=\pi_a^{\prime}$. Then,
\[ \left(u_d \hdots u_{k+1} w_{k+1}\right) \left( w_{k+1} \pi_a^{\prime} w_{k+1}\right) \left(w_{k+1} u_{k+1} \hdots u_d\right)=\pi_b.\]

The DCJ distance between $\left( w_{k+1} \pi_a^{\prime} w_{k+1}\right)$ and $\pi_b$ is $d-k$. Therefore $u_d \hdots u_{k+1} w_{k+1}$ can be re-written as a product of $d-k$ transpositions, say $v_{d-k} \hdots v_1$. Now $v_{d-k} \hdots v_1 w_{k+1} u_k \hdots u_1$ is an expression of length $d$ equal to $g$ and prefix $u_k \hdots u_1$ that is not in $S_k(g)$. This contradicts the definition of $S_k(g)$.

By repeating this argument, we can conclude that there exists some element in $S(g)$ that gives the same sorting scenario as $h$.

Thus the number of sorting scenarios is equal to $\left \vert S(g)\right \vert= (d+1)^{d-1}$.
\qed
\end{proof}

%%=============================================%
\section{Conclusions and future work}
%%=============================================%

The double cut and join operator has been a major step forward for the study of genome rearrangements, because it is very general, allowing numerous operations on multi-chromosomal genomes, and in addition has a very simple length formula with which to calculate genomic distance.  In this paper, we have shown how this operator may be described group-theoretically, and derived a correspondingly simple length formula independently of prior results.  The length formula given in Theorem~\ref{thm:dcjDisTotal} requires only the ability to write each genome as a permutation and to multiply such permutations. We are also able to provide a simple construction for an optimal sorting scenario in particular instances of the problem. Translating the model into algebra allows us to exploit established results in group theory, a field with over a century of development.  The proof of Theorem~\ref{thm:sortingScen} is an example of this, relying as it does on a combinatorial group theory result from the 1950s.

The use of group theory to model rearrangements provides a natural context in which to study alternative assumptions about the rearrangement processes.  As pointed out in~\citet{egrinagy2013group} and~\citet{francis2013algebraic}, different assumptions about the processes gives rise to questions about length functions in different groups, or length functions with respect to different generating sets. A group-theoretic model may also provide an avenue for investigating additional operations that are not captured by the DCJ.

\bibliographystyle{plain}

\appendix
\normalsize

\section{Some results about Symmetric Groups} \label{sec:gt_primer}
This paper uses some standard results on symmetric groups that we collect here for ease of reference.  More details on these results can be found in many undergraduate group theory textbooks, for example~\cite{fraleigh2003first}.

A permutation is a bijection from a set $S$ to itself. $S$ is usually taken to be a set of natural numbers $\mathbf{n}=\{1,2,\hdots,n\}$. 
A permutation can be written by specifying the value of the map on all the points. 
 
 \eg
 \[
 \pi = \bigl(\begin{smallmatrix}
 1 & 2 & 3 & 4 & 5 & 6 \\
 3 & 4 & 1 & 6 & 2 & 5
 \end{smallmatrix}\bigr)
 \]
 is a permutation on the set $\{1,2,3,4,5,6\}$ that sends 1 to 3, 2 to 4, etc.
 The set of all permutations on the set $\mathbf{n}$ forms a {\it group} called the symmetric group and denoted by $S_{n}$.

 \subsection{Permutation multiplication}
 Since permutations are simply bijective functions, permutation multiplication is function composition. That is, to find the image of $i$ in the product $\pi_2\pi_1$,
 we do $\pi_2 \left(\pi_1(i)\right)$. 
 
 \eg Let 
 $\pi_1 = \bigl(\begin{smallmatrix}
 	1 & 2 & 3 & 4 & 5  \\
 	3 & 4 & 1 & 5 & 2 
 \end{smallmatrix}\bigr)$ and 
 $\pi_2= \bigl(\begin{smallmatrix}
 1 & 2 & 3 & 4 & 5  \\
 2 & 1 & 4 & 3 & 5 
 \end{smallmatrix}\bigr)$. The image of $1$ in the product $\pi_2\pi_1$ is $\pi_2 \left(\pi_1(1)\right)=\pi_2(3)=4$. So for each $i$, we have to ``follow the string'' -- $\pi_1$ send $i$ to $j$, $\pi_2$ sends $j$ to $k$, so $i$ gets sent to $k$ by $\pi_2\pi_1$. 
 
 \[\bigl(\begin{smallmatrix}
 1 & 2 & 3 & 4 & 5  \\
 2 & 1 & 4 & 3 & 5 
 \end{smallmatrix}\bigr) 
 \bigl(\begin{smallmatrix}
 1 & 2 & 3 & 4 & 5  \\
 3 & 4 & 1 & 5 & 2 
 \end{smallmatrix}\bigr)=\bigl(\begin{smallmatrix}
 1 & 2 & 3 & 4 & 5  \\
 4 & 3 & 2 & 5 & 1 
 \end{smallmatrix}\bigr).
 \]
 
 \subsection{Inverse of a permutation}
 Informally, a permutation $\pi \in S_{n}$ scrambles the elements of $\mathbf{n}$. The inverse of $\pi$ is the permutation that ``undoes'' the scrambling. Formally we define the {\it identity} permutation $\iota$ to be the permutation that maps $i$ to $i$ for all $i \in \mathbf{n}$.
 
 \begin{definition}[Inverse]
 Let $\pi \in S_{n}$. Then the inverse of $\pi$ is the permutation $\pi^{-1}$ such that
 	\[\pi \pi^{-1}=\iota\quad\text{and}\quad \pi \pi^{-1}=\iota.\]
 \end{definition}
 
 If $\pi^{-1}$ is the inverse of $\pi$ then $\pi$ is the inverse of $\pi^{-1}$. That is, $(\pi^{-1})^{-1}=\pi$. In general,
 \( (\pi_1\pi_2)^{-1}=\pi_2^{-1}\pi_1^{-1}.\)

 \subsection{Cycles and cycle decomposition}
 
 For a permutation $\pi \in S_{n}$,  if we repeatedly apply $\pi$ to any $i \in \mathbf{n}$,
 \[i  \op{\pi} \pi(i)  \op{\pi} \pi^2(i) \hdots ,\]
 we must eventually (say after $k$ steps) reach $i$ again since $\mathbf{n}$ is a finite set. If there is some $j \in \mathbf{n}$ which does not occur in this sequence, then we can form a similar sequence for $j$, and keep doing this until every element of $\mathbf{n}$ occurs in some sequence.

 \begin{definition}[Cycle]
 Let $i_1,i_2,\hdots i_k$ be $k$ distinct integers in $\mathbf{n}$. A cycle $\pi_c$ written as $(i_1,i_2,\hdots ,i_k)$ is a permutation in $S_n$ defined as
 \[\pi_c(i_s):=
  \begin{cases}
    i_{s+1} & \text{ if } i_s \in \{i_1,i_2,\hdots i_k\}\\
 	i_s & \text{ otherwise.}
  \end{cases}
 \]
 \end{definition}
 
 A $2$-cycle is a cycle of length $2$. That is, $\pi=(i,j)$ means that $\pi(i)=j,\pi(j)=i$ and $\pi(k)=k$ if $k \neq i,j$. A cycle of length $2$ is also called a {\it transposition}.
 
 Two cycles are said to be disjoint if they have no elements in common.
 
 \begin{theorem}
 Any permutation $\pi \in S_{n}$ can be written as a product of disjoint cycles.
 \end{theorem}
 \eg Let $
 \pi = \bigl(\begin{smallmatrix}
 1 & 2 & 3 & 4 & 5 & 6 \\
 3 & 4 & 1 & 6 & 2 & 5
 \end{smallmatrix}\bigr)
 $.  $\pi$ can be written as
 \[\pi=(1,3)(2,4,6,5).\]
 
This way of writing a permutation is referred to as {\it cycle notation}.  There is a unique way of writing a permutation as a product of disjoint cycles, up to the ordering of the cycles (they commute) and cyclic equivalence of each cycle (e.g. $(1,2,3)=(2,3,1)=(3,1,2)$).  Since the sizes of the disjoint cycles will always add to $n$ (including if necessary some 1-cycles), we can define the cycle type as follows.

\begin{definition}[Cycle type]
The cycle type of a permutation $\pi$ is the partition $\lambda\vdash n$ whose components are the sizes of the cycles in the disjoint cycle decomposition of $\pi$.
\end{definition}

\eg The cycle type of $\pi=(1,3)(2,4,6,5)$ is $(4,2)$ since it has one cycle of length $2$ and one cycle of length $4$.
\subsection{Conjugation}
\begin{definition}
Let $\pi,g \in S_{n}$. The conjugate of $\pi$ by $g$ is defined to be the permutation $g \pi g^{-1}$, and we say that $\pi$ and $g\pi g^{-1}$ are {\it conjugate} permutations. 
\end{definition}

\begin{theorem}
Let $\pi_1$ and $\pi_2$ be permutations on the set $\mathbf{n}$, then $\pi_1$ and $\pi_2$ are conjugate in $S_{n}$ if and only if they have the same cycle type.
\end{theorem}

\subsection{Permutation as product of transpositions}
\begin{theorem}
Any permutation $\pi \in S_{n}$ can be written as a product of transpositions.
\end{theorem}

\eg The permutation $\pi=(1,3)(2,4,6,5)$ can be written as 
\[\pi=(1,3)(2,4,6,5)=(1,3)(2,5)(2,6)(2,4).\]

While the decomposition of a permutation into a product of disjoint cycles is unique, the decomposition of a permutation into a product of transpositions is not unique. However the number of transpositions used must be either always be even, or always be odd. 

\begin{theorem}
A permutation $\pi \in S_n$ can be expressed as a product of either an even number of transpositions or an odd number of transpositions, but not both.
\end{theorem}

\begin{definition}
A permutation is said to \emph{even} if it can written as a product of an even number of transpositions. Otherwise it is said to be an \emph{odd} permutation.
\end{definition}

\end{document}